\documentclass[oneside, reqno]{amsart}
\usepackage{graphicx} 
\usepackage{float}
\usepackage{amsmath}
\usepackage{amsthm}
\usepackage[a4paper]{geometry}

\usepackage[hidelinks]{hyperref}
\usepackage{amssymb}
\usepackage{tikz-cd} 
\usepackage{quiver}
\usepackage{mathrsfs} 
\usepackage{enumitem}
\setlist{nosep}

\newtheorem{theorem}{Theorem}
\newtheorem*{thm*}{Theorem}

\newtheorem{lemma}[theorem]{Lemma}

\theoremstyle{definition}

\newtheorem{conjecture}[theorem]{Conjecture}

\numberwithin{equation}{section}

\newcommand{\N}{\mathbb{N}}

\title{A Hajnal-Szemerédi theorem for skewed colorings}

\author[M. Birken]{Mathis Birken}
\address[M. Birken]{Institut f\"ur Informatik, Universit\"at Heidelberg, Heidelberg,
Germany
}
\email{mathis.birken@uni-heidelberg.de}

\date{\today}
\begin{document}

\begin{abstract}
    The Hajnal–Szemerédi theorem guarantees an equitable \((r+1)\)-coloring of every graph of maximum degree at most \(r\). We prove a version of this result for more general prescribed color class sizes, subject only to a natural upper bound.
\end{abstract}

\maketitle

\section{Introduction}

In this note, we prove a conjecture of Kuchukova, Perkins and Povill, namely that a variant of the Hajnal-Szemerédi theorem holds for \textit{skewed} or \textit{prescribed} colorings. More precisely, given $n \in \N_0$ and a vector $\vec n= (n_1, \dots, n_d) \in \N^d$ of positive integers summing to $n$, an \textit{$\vec n$-coloring}, or a \textit{prescribed coloring}, of a graph $G$ on $n$ vertices is a proper coloring of the vertices of $G$ such that for each $i \in [d]$, there are exactly $n_i$ vertices in color class $i$. We prove the following theorem.

\begin{theorem} \label{main_theorem}
    Let $G$ be a graph on $n$ vertices with $\Delta(G) \le r \in \N_0$. If $\vec n$ is a color vector whose coordinates satisfy $n_i \le \lfloor {n / (r +1)} \rfloor$ for all $i \in [d]$, then there exists an $\vec n$-coloring of $G$.
\end{theorem}

Theorem~\ref{main_theorem} implies the equitable case as follows. Let $G$ be a graph on $n$ vertices with maximum degree at most $r$. Then $G$ admits an equitable $(r+1)$-coloring: if $n$ is a multiple of $r+1$, this follows directly from Theorem~\ref{main_theorem}. Otherwise, apply the theorem to $G \sqcup K_p$, where $p \in [r]$ is chosen so that $n + p$ is a multiple of $r+1$. The vertices of $K_p$ receive distinct colors, so restricting the coloring to $G$ yields color classes whose sizes differ by at most one.

The Hajnal-Szemerédi theorem on equitable colorings was originally proved in \cite{HajnalSzemeredi1970}. The prescribed version stated in Theorem~\ref{main_theorem} appears as Conjecture 1.8 in \cite{kuchukova2026}. Our proof is obtained by adapting the later proof of the Hajnal-Szemerédi theorem due to Kierstead and Kostochka \cite{kierstead_kostochka}. The following is a self-contained account of their proof with the necessary modifications. 

\section{Main proof}

Let $G$ be a graph as in Theorem~\ref{main_theorem}. Fix some color vector $\vec n$. We say that a proper coloring $f$ of $G$ is an \textit{almost-$\vec n$-coloring}, or an \textit{almost-prescribed coloring} if there is a \textit{large} class $V^+ = V^+(f)$ and a \textit{small} class $V^-= V^-(f)$ which differ by 1 compared to the prescribed sizes.
More precisely, if $V_1, \dots, V_d$ are the color classes of $f$, with $V^+ = V_\alpha$, $V^- = V_\beta$ and $\alpha \neq \beta$, then we demand that $|V^+| = n_\alpha + 1$ and $|V^-| = n_\beta - 1$; and $|V_i| = n_i$ for all other $i \in [d]$. We will investigate how one can transform an almost prescribed coloring into a prescribed coloring. Again, we will follow \cite{kierstead_kostochka} very closely and any reader familiar with the original paper is invited to focus only on the changes compared with the equitable case.

Given an almost-prescribed coloring $f$, we consider the following auxiliary digraph $H = H(G, f)$: the vertices of $H$ are the color classes of $f$ and for distinct color classes $V, W$, there is an arc $V\to W \in E(H)$ if and only if some vertex $y \in V$ is \textit{movable} to $W$, i.e. $y$ has no neighbors in $W$.
We call $W \in V(H)$ \textit{accessible} if there exists a directed path from $W$ to $V^-$ in $H$. We write $\mathscr A = \mathscr A(f)$ for the set of accessible classes. Trivially, $V^- \in \mathscr A$. Moreover, we set $A := \bigcup \mathscr A \subseteq V(G)$. Similarly, we denote by $\mathscr B := V(H) \setminus \mathscr A$ the set of inaccessible color classes and write $B := V(G) \setminus A$. The easiest case is the following.

\begin{lemma} \label{first_lemma}
    If $f$ is an almost prescribed coloring of $G$ whose large class $V^+$ is accessible, then $G$ admits a prescribed coloring.
\end{lemma}

\begin{proof}
    If $V_1\to \cdots \to V_k$ is a path from $V_1=V^+$ to $V_k=V^-$, then for each $j \in [k-1]$ there exists a vertex $y_j \in V_j$ which no neighbor in $V_{j+1}$. Moving all vertices $y_j$ simultaneously produces a prescribed coloring.
\end{proof}

Assume from now on that $V^+$ is inaccessible, so $V^+ \subseteq B$. We set $s = \lfloor {n / (r +1)} \rfloor$. Moreover, let $a := |\mathscr A|$ and $q := r + 1 - a$. Since $r + 1 \le d$ (as $s(r + 1) \le n = \sum_{i= 1}^d n_i\le sd$), we have $|\mathscr B| = d - a \geq q$. Since no $y \in B$ is movable to any of the classes in $\mathscr A$, we have 
\begin{equation} \tag{1} \label{eq_d_B_y}
    d_A(y) \ge a \; \text{ and hence } \; d_B(y) \le q - 1.
\end{equation}
In particular, $q \ge 1$. Note that $|V^-| \le s - 1$, so we have $|A| \le as - 1$. Consequently,
\begin{equation} \tag{$\star$} \label{eq_size_B}
    |B| = n - |A| \ge (r+1)s - as + 1 = qs + 1.
\end{equation}

Moreover, if $A = V^-$, then by (\ref{eq_d_B_y}) we have 
\[ |E(A, B)| \le r (s- 1) < rs + 1 \le |B|,\]
which is a contradiction, since each vertex $y \in B$ must have at least one neighbor in $A$. It follows that $a \ge 2$.

We now choose a subset $\mathscr A' \subseteq \mathscr A$ (and set $A' := \bigcup \mathscr A'$) as follows. For $W \in \mathscr A$ consider the set $\mathscr T_W \subseteq \mathscr A \setminus \{W\}$ consisting of the classes $V$ for which every directed path from $V$ to $V^-$ in $H$ contains $W$. In other words, removing $W$ cuts off $\mathscr T_W$ from $V^-$. A color class $W$ is called \textit{initial} if $\mathscr T_W = \emptyset$. We now choose a non-initial $U$ such that $\mathscr A' := \mathscr T_U \neq \emptyset$ is inclusion-minimal. This is possible as $V^-$ is non-initial. By minimality, all color classes in $\mathscr A'$ must be initial. 
Let $t := |\mathscr A'|$. Note that no class in $\mathscr A'$ has a vertex movable to any class in $(\mathscr A \setminus \mathscr A')\setminus \{U\}$. Thus every $x \in A'$ satisfies 
\begin{equation} \tag{2} \label{ineq_d_A_x}
    d_A(x) \ge a - t - 1.
\end{equation}

We introduce the following terminology: an edge $zy$ with $z \in W \in \mathscr A'$ and $y \in B$ is called a \textit{solo edge} if $N_W(y) = \{z\}$, that is, $z$ is the only neighbor of $y$ belonging to $W$. The endpoints of a solo edge are called \textit{solo vertices}, and each endpoint is called a \textit{special neighbor} of the other. We let $S_z \subseteq B$ and $S^y \subseteq A'$ denote the sets of special neighbors of $z \in A'$ and $y \in B$, respectively. For each $y \in B$, at most $r - (a + d_B(y))$ color classes in $\mathscr A$ can contain more than one neighbor of $y$. This implies
\begin{equation} \tag{3} \label{ineq_size_Sy}
    |S^y| \ge t - q + 1 + d_B(y).
\end{equation}

\begin{lemma} \label{second_lemma}
    If there exists a color class $W \in \mathscr A'$ such that no solo vertex in $W$ is movable to a class in $\mathscr A \setminus \{W\}$, then $q + 1 \le t$. Moreover, in this case every vertex $y \in B$ is solo.
\end{lemma}

\begin{proof}
    Let $S \subseteq W$ denote the set of solo vertices and let $D := W \setminus S$. If a vertex in $B$ is adjacent to exactly one vertex in $W$, then by definition it is adjacent to a vertex in $S$. Hence, any vertex in $B \setminus N_B(S)$ must have at least two neighbors in $W$. We therefore obtain the lower bound
    \[2|B| - |N_B(S)| = |N_B(S)| + 2|B \setminus N_B(S)| \le  |E(W, B)|.\]
    Since no solo vertex in $W$ is movable to a class in $\mathscr A \setminus \{W\}$, we have $d_B(z) \le q$ for all $z \in S$, so $|N_B(S)| \le q |S|$. Combining this observation with $|S| + |D| \le s$ and (\ref{eq_size_B}), we obtain
    \[qs + q|D| + 2  \le 2(qs + 1) - q|S| \le  |E(W, B)|.\]
    On the other hand, by (\ref{ineq_d_A_x}) we have $d_B(x) \le q + t$ for all $x \in A'$. This yields the upper bound
    \[|E(W, B) | \le q|S| + (t + q)|D| \le qs + t|D|.\]
    Combining the two bounds, obtain $q|D| + 2 \le t|D|$, which implies $q + 1 \le t$. Moreover, by (\ref{ineq_size_Sy}), we obtain $|S^y| \ge 2$ for every $y \in B$, implying that all vertices of $B$ are solo.
\end{proof}

\begin{lemma} \label{third_lemma}
    There exists a color class $W \in \mathscr A'$ and a solo vertex $z \in W$ such that either $z$ is movable to a class in $\mathscr A \setminus \{W\}$ or $z$ has two nonadjacent special neighbors in $B$. 
\end{lemma}
\begin{proof}
    Suppose not. Then for every solo vertex $z \in A'$, $S_z$ is a clique and no solo vertex in $A'$ can be moved to a different class in $\mathscr A$. In particular, every class $ W \in \mathscr A'$ satisfies the hypothesis of Lemma~\ref{second_lemma}. It follows that every vertex $y \in B$ is solo. We define a weight function $\mu$ on $E(A', B)$ given by 
    \[\mu(xy) = \begin{cases}
        \frac{q}{|S_x|} & \text{if } xy \text{ is solo,}\\
        0 & \text{if } xy \text{ is not solo.}\\
    \end{cases}\]
    For $C,D\subseteq V(G)$, write
    $\mu(C,D):=\sum_{xy\in E(C,D)}\mu(xy),$
    and abbreviate $\mu(\{x\},D)=\mu(x,D)$.
    Then if $z \in A'$ is solo, we have $\mu(z,B) = q$, otherwise $\mu(z,B) = 0$. We obtain $\mu(A', B) \le |A'|q \le qst$. 
    
    Now consider $y \in B$. Choose $z \in S^y$ such that $|S_z|$ is maximal and set $c_y := |S_z|$. Since $S_z \subseteq B$ is a clique containing $y$, \eqref{eq_d_B_y} gives $c_y - 1\le d_B(y) \le q - 1$, and hence $c_y \le q$. Combining this with (\ref{ineq_size_Sy}), we obtain
    \[\mu(A', y) = \sum_{x \in S^y} \frac{q}{|S_x|} \ge |S^y| \frac{q}{c_y} \ge (t - q + c_y) \frac{q}{c_y} \ge (t-q)\frac{q}{c_y} + q \ge t,\]
    where in the last step we used that $t - q \ge 1$ by Lemma~\ref{second_lemma}. Together with (\ref{eq_size_B}), this yields
    \[\mu(A', B) \ge t|B| \ge t(qs + 1) > qst \ge \mu(A', B),\]
    which is a contradiction.
\end{proof}

\begin{proof}[Proof of Theorem~\ref{main_theorem}]
    We proceed by induction on the number of edges of $G$. If $G$ has no edges, the statement clearly holds. We now assume that the theorem holds for all graphs with fewer edges, for every degree bound $r$ and every color vector.
    Given an edge $e= xy \in E(G)$, the induction hypothesis gives a prescribed coloring of $G-e$. Now, if $x$ and $y$ received different colors, then the same coloring is a prescribed coloring of $G$. Otherwise, if $x$ and $y$ lie in the same color class $V$, is possible to move $x$ to a different color class $W$ since $d(x) \le r$ and $r+1 \le d$. This yields an almost prescribed coloring $f$ of $G$ with $V^+ = W \cup \{x\}$ and $V^- = V\setminus \{x\}$. We will now show that $G$ has a prescribed coloring by a secondary induction on $q = q(f)$.

    If $V^+ \in \mathscr A$, we are done by Lemma~\ref{first_lemma}. In particular, this handles the base case $q \le 0$, since if $B \neq \emptyset$, then (\ref{eq_d_B_y}) implies that $a \le r$, so $q \ge 1$. Otherwise, Lemma~\ref{third_lemma} guarantees the existence of a solo vertex $z \in W \in \mathscr A'$ such that one of the following cases holds.

    \vspace{6pt}
    \textit{Case 1:} $z$ is movable to some class $X \in \mathscr A \setminus \{W\}$. Choose $y_1 \in S_z$. Write $B^- := B \setminus \{y_1\}$ and $A^+ = A \cup \{y_1\}$. Move $z$ to $X$ and $y_1$ to $W \setminus \{z\}$; denote by $\varphi$ the coloring of $G[A^+]$ thus obtained.  Clearly, $W \neq V^-$. Since $X$ is accessible with respect to $f$ and $W \in \mathscr A'$ is initial, there exists a (possibly trivial) directed path from $X$ to $V^-$ in $H(G, f)$ which avoids $W$. Such a path can now be used as in Lemma~\ref{first_lemma} to produce a coloring $\varphi'$ of $G[A^+]$ for which all the classes in $\mathscr A$ have the prescribed sizes.
        
    We now apply the primary induction hypothesis. By (\ref{eq_d_B_y}), the induced graph $G' := G[B^-]$ has maximum degree at most $q-1$. Moreover, $zy_1 \in E(G)$ but $zy_1 \notin E(G')$, and hence $G'$ has fewer edges than $G$. Let $\vec n_{\mathscr B}$ denote the restriction of $\vec n$ to the coordinates indexed by the classes in $\mathscr B$. Since 
    \[
    \left \lfloor \frac{|B^-|}{(q-1)+1} \right\rfloor \ge \left \lfloor \frac{(qs + 1) - 1}{q} \right \rfloor  = s,
    \]
    we may indeed apply the primary induction hypothesis to $G'$, producing an  $\vec n_{\mathscr B}$-coloring $g$ of $G'$. Combining $\varphi'$ and $g$ produces the desired prescribed coloring of $G$. 
        
    \vspace{6pt}
    \textit{Case 2:} $z$ is not movable to any class in $\mathscr A \setminus \{W\}$ and $z$ has two nonadjacent special neighbors $y_1, y_2 \in S_z$. Define $A^+$ and $B^-$ as above. As in Case 1, the primary induction hypothesis gives a prescribed coloring $g$ of $G[B^-]$.
    Note that $d_{A^+}(z) = d_A(z) +1\ge a$, so $d_{B^-}(z) \le r - a = q-1$. Since $|\mathscr B| \ge q$, we can therefore move $z$ to a $g$-color class $Y \subseteq B^-$, producing a coloring $g'$ of $B^* := B^- \cup \{z\}$. Moreover, move $y_1$ to $W$ to obtain a coloring $\psi$ of $A^* := V(G) \setminus B^*$. Combine the colorings $\psi$ and $g'$ into an almost-prescribed coloring $h$ of $G$.
    Observe that we have $A^* \subseteq A(h)$: every class $V \in \mathscr A \setminus \{W\}$ remains accessible since $W \in \mathscr A'$ is initial, so $H(G, f)$ contains a path from $V$ to $V^-$ which avoids $W$.  Finally, the modified class $W^* := (W \cup \{y_1\}) \setminus \{z\}$ is accessible with respect to $h$ because any path from $W$ to $V^-$ in $H(G, f)$ survives: the first arc has a witness different from $z$, as we assumed that $z$ is not movable in $\mathscr A$.
    
    Finally, in this case the vertex $y_2$ is movable to the class $W^*$, so the $h$-color class of $y_2$ is accessible. Thus $q(h) < q(f)$, and so the secondary induction hypothesis gives a prescribed coloring of $G$.
\end{proof}

\section{Future work}

We conjecture that if $n$ is not a multiple of $r+1$, then the color vector may contain a certain number of larger classes. This would generalize the equitable case mentioned above in a natural way.

\begin{conjecture}
    Let $n = s(r+1) + m$, where $m \in [r]$. Let $\vec n = (n_1, \dots, n_d)$ be a color vector with $n_i = s+1$ for at most $m$ coordinates and $n_i \leq s$ for all other $i$.
    Then any $n$-vertex graph~$G$ with $\Delta(G) \leq r$ has an $\vec n$-coloring.
\end{conjecture}

For instance, the graph $G = sK_{r+1} \sqcup K_m$ admits exactly the colorings described by the conjecture.

\vspace{6pt}

\textbf{AI statement:} The author used AI tools, specifically ChatGPT, for proofreading and to assist in developing the conjecture.

\bibliographystyle{alpha}
\bibliography{literature}

@article{kierstead_kostochka, title={{A Short Proof of the Hajnal–Szemerédi Theorem on Equitable Colouring}}, volume={17}, DOI={10.1017/S0963548307008619}, number={2}, journal={Combinatorics, Probability and Computing}, author={Kierstead, H. A. and Kostochka, A. V.}, year={2008}, pages={265–270}}

@misc{kuchukova2026,
      title={Sampling Colorings with Fixed Color Class Sizes}, 
      author={Kuchukova, A. and Perkins, W. and Povill, X.},
      year={2026},
      eprint={2603.08259},
      archivePrefix={arXiv},
      primaryClass={math.CO},
      url={https://arxiv.org/abs/2603.08259}, 
      note={Available at \url{https://arxiv.org/abs/2603.08259}}
}

@incollection{HajnalSzemeredi1970,
  author    = {Hajnal, A. and Szemer{\'e}di, E.},
  title     = {{Proof of a conjecture of P. Erd{\H{o}}s}},
  booktitle = {Combinatorial Theory and its Application},
  pages     = {601--623},
  publisher = {North-Holland},
  address = {London},
  year      = {1970}
}

\end{document}